\documentclass{article}
\usepackage{graphicx} % Required for inserting images
\usepackage{multicol}
\usepackage{amsfonts,amssymb,amsmath}
\usepackage{multicol}
\usepackage{tikz, pgfplots}
\usepackage{tikz-cd}
\usepackage{cancel}
\usepackage{amsthm}
\usepackage{hyperref}
    \hypersetup{
    colorlinks=true,
    linkcolor=blue,      
    urlcolor=blue,
    }

\usepackage{mathtools}
\usepackage{ytableau}

\newtheorem{theorem}{Theorem}

\newtheorem{lemma}[theorem]{Lemma}
\newtheorem{proposition}[theorem]{Proposition}
\newtheorem{remark}[theorem]{Remark}

\newcommand{\col}{ \text{\normalfont{Col}} }
\newcommand{\row}{ \text{\normalfont{Row}} }
\newcommand{\swap}{ \text{\normalfont{swap}} }
\newcommand{\wmnop}{ w_\text{\normalfont{MNOP}} }
\newcommand{\wjack}{ w_\text{\normalfont{Jack}} }

\usepackage{caption}
\usepackage{subcaption}
\usepackage{wasysym}
\usepackage[utf8]{inputenc}
\usepackage[T1]{fontenc}

\title{{Jack combinatorics of the equivariant edge measure}}
\author{Kyla Pohl and Ben Young \\ \small{University of Oregon}}
\date{\today}

\usepackage{todonotes}

\makeatletter
\newcommand\ackname{Acknowledgements}
\if@titlepage
  \newenvironment{acknowledgements}{%
      \titlepage
      \null\vfil
      \@beginparpenalty\@lowpenalty
      \begin{center}%
        \bfseries \ackname
        \@endparpenalty\@M
      \end{center}}%
     {\par\vfil\null\endtitlepage}
\else
  \newenvironment{acknowledgements}{%
      \if@twocolumn
        \section*{\abstractname}%
      \else
        \small
        \begin{center}%
          {\bfseries \ackname\vspace{-.5em}\vspace{\z@}}%
        \end{center}%
        \quotation
      \fi}
      {\if@twocolumn\else\endquotation\fi}
\fi
\makeatother

\begin{document}

\maketitle

\begin{abstract}
    We study the equivariant edge measure: a measure on partitions which arises implicitly in the edge term in the localization computation of the Donaldson-Thomas invariants of a toric threefold.  We combinatorially show that the equivariant edge measure is, up to choices of convention, equal to the Jack-Plancherel measure.

\end{abstract}

\section{Introduction}\label{introductionmotivation}

We study a measure on integer partitions called the \emph{equivariant edge measure}.  Our work is motivated by a problem in plane partition combinatorics that is well-known to enumerative geometers, but perhaps less well known to combinatorialists: the \emph{equivariant vertex measure}. We first describe this motivating problem combinatorially.

\subsection{Motivation: The Equivariant Vertex Measure}

In their 2006 papers (\cite{MNOP1, MNOP2}), Maulik, Nekrasov, Okounkov, and Pandharipande (hereafter, MNOP) compute the Donaldson-Thomas partition function of a threefold with a torus action; they use Atiyah-Bott localization to reduce the computation to the torus fixed loci of the action.  One result of their calculations is an exotic measure $\mathsf{w}$ on plane partitions, called the \emph{equivariant vertex measure}.  This measure has has three parameters $u, v, w$ coming from the torus action.  The equvariant vertex measure of a plane partition $\pi$ is a certain rational function in $u, v, w$, defined as follows: first, we temporarily choose 3 other formal parameters, $r, s, t$.  We write 
$$Q = \sum_{(i,j,k) \in \pi} r^is^jt^k \qquad \text{and} \qquad \overline{Q} = \sum_{(i,j,k) \in \pi} r^{-i}s^{-j}t^{-k},$$
where both sums are over the boxes in the 3D Young diagram corresponding to $\pi$.  Next, write $$
F = Q - \frac{\overline{Q}}{r s t} + Q\overline{Q}\frac{(1-r)(1-s)(1-t)}{rst}
=: \sum_{i,j,k}c_{ijk}r^{i}s^{j}t^{k}.
$$
One checks that $F$ is a Laurent polynomial with no constant term, so it makes sense to ``swap the roles of addition and multiplication'' in this formula. 
As far as we understand the geometry, the expression for $F$ appears in the computation of the character of the torus action, and the ``swap'' happens in the process of the localization calculation of the Donaldson-Thomas invariants from this character; however, for the purposes of this note, the swap is a purely formal operation.
We may finally define the \textit{equivariant vertex measure}:
$$\mathsf{w}(\pi) = \prod_{i,j,k}(i u + j v + k w)^{- c_{ijk}}.$$  
One of the results of \cite{MNOP2} is a closed form for the partition function of the equivariant vertex measure. Namely, let $|\pi|$ denote the number of boxes in the 3D young diagram for $\pi$. Then
\begin{align}\label{eq:Z}
    Z := \sum_{\pi} \mathsf{w}(\pi) q^{|\pi|}= M(q)^{- \frac{(u+v)(v+w)(w+u)}{uvw}}, \qquad \text{where} \qquad M(q) = \prod_{i \geq 1} \left( \frac{1}{1-q^i}\right)^i.
\end{align}
is Macmahon's generating function for plane partitions \cite{Macmahon}.  Note also that our notation differs from that in~\cite{MNOP1,MNOP2}: we use different variable names, and we use $q$ in place of $-q$.

Such a closed formula as \eqref{eq:Z}, in principle, puts $\mathsf{w}$ into the class of \emph{exactly solved} models in statistical mechanics.  That is, $\mathsf{w(\pi)}/Z$ is an explicit \emph{probability} measure on plane partitions $\pi$.  One might hope to say essentially anything about $\mathsf{w(\pi)}/Z$ -- the limiting shape of a typical plane partition, for instance, or the probability of finding a particular local configuration of boxes in its 3D Young diagram.  Indeed, although the definition of $\mathsf{w}(\pi)$ is completely combinatorial, the proof of the formula for $Z$ is intricate and geometric.  One would like a purely combinatorial proof.  Unfortunately, all of these hopes are currently well out of reach, based on current combinatorics knowledge.  Note that in the so-called \emph{Calabi-Yau specialization}, where $u+v+w = 0$, everything simplifies directly:  we have $\mathsf{w}(\pi) \equiv 1$ and $Z = M(q)$, recovering Macmahon's formula; random plane partitions under this measure have been extensively studied by many authors (see for instance \cite{kenyon-okounkov-sheffield, okounkov-reshetikhin} for this precise problem, or \cite{gorin} for a survey).

\subsection{The Equivariant Edge Measure}
\label{sec:intro_dt}
    
Having introduced $\mathsf{w}(\pi)$ and proclaimed it to be hard to study, we now turn to a two-dimensional version of this problem. The subject of this note is the analogue of the equivariant vertex measure for ordinary partitions which implicitly arises in~\cite{MNOP1}. More specifically, it arises in the ``edge term'', coming from a torus-invariant \emph{line} in the localization calculation.
 Given a partition $\lambda$, as in \cite[Section 4.8]{MNOP1} define generating functions
\begin{align}
    Q(\lambda) = \sum_{(i,j)\in\lambda} r^is^j \qquad \text{ and } \qquad \overline{Q}(\lambda) = \sum_{(i,j)\in\lambda} r^{-i}s^{-j}
\end{align}
where the sums are taken over the coordinates of all cells in $\lambda$. Then, following~\cite[Equation (12)]{MNOP1}, define 
\begin{align}
    F(\lambda) = - Q(\lambda) - \frac{\overline{Q}(\lambda)}{rs} + \frac{Q(\lambda)\overline{Q}(\lambda)(1-r)(1-s)}{rs}.
\end{align}
We then swap the role of addition and multiplication much as before, though for convenience in stating the main result, we include a minus sign on $v$, and omit one in the exponent (as compared with the definition of $\mathsf{w})$.
Given an index set $A$ and a Laurent polynomial $G = \sum_{(i,j) \in A}c_{i,j}  r^{i} s^{j}$ in the variables $r$ and $s$ with no constant term, we define the swap operation as follows:
\begin{align}
\label{eqn:swaperation}
    \text{swap}(G) = \text{swap}\left(\sum_{(i,j) \in A}c_{i,j} r^{i} s^{j} \right) = \prod_{(i,j) \in A}(iu-jv)^{c_{i,j}}
\end{align}
We then define the \textit{equivariant edge measure} $\wmnop$ to be the swap operation applied to $F(\lambda)$:
\begin{align}
\label{eqn:def of wmnop}
    \wmnop(\lambda) := \swap(F(\lambda)).
\end{align}

\subsection{The Jack Plancherel Measure}

We now introduce a second, \emph{a priori different}, measure on Young diagrams. The \emph{Jack-Plancherel measure} was first studied in~\cite{kerov1997anisotropic}; it is a probability measure on partitions of $n$, defined by
\begin{align}\label{eq:jackplancherel}
    \wjack(\lambda) := \frac{1}{\prod_{(i,j)\in\lambda} h^*(i,j) h_*(i,j)}.
\end{align}
The upper and lower hook lengths $h^*$ and $h_*$ are in turn defined carefully in Section~\ref{backgroundpartitionsandhooklengths}. They are two perturbations of the standard hook length of a cell in a Young diagram, and they depend on a single parameter $\beta$ (which we homogenize using two parameters $u,v$).

The Jack Plancherel measure is described in~\cite{borodin-olshanski} as a natural perturbation of ordinary Plancherel measure.  Okounkov~\cite{okounkov} also advanced the viewpoint that the Plancherel measure is a close analogue of the Gausiann Unitary ensemble from random matrix theory and shares many of the same universal asymptotics; in this framework the Jack-Plancherel measure should have many of the universal asymptotics of the random-matrix-theoretic $\beta$ ensembles.  Indeed, the bulk asymptotics were later computed in~\cite{Do_ga_2016} and the edge asymptotics in~\cite{guionnet2017rigidityedgeuniversalitydiscrete}.  See also \cite{BenDaliDolega2023} for a recent interpretation of Jack character values (of which the Jack-Plancherel measure is one) in terms of enumerations of maps on surfaces.

\subsection{Main Result and Outline}
Our main result is that the equivariant edge measure $\wmnop$ is the Jack Plancherel measure. 
\begin{theorem}\label{mainresult}
    The Jack Plancherel measure of a partition $\lambda$ is the same as the equivariant edge measure of $\lambda$ up to a sign, i.e. 
    \begin{align}
        \wjack(\lambda) = -\wmnop(\lambda).
    \end{align}
\end{theorem}

It is no surprise to see the Jack-Plancherel measure arising in enumerative geometry.  For instance, Okounkov~\cite{okounkov} described how the Jack-Plancherel measure appears in Seiberg-Witten theory, and in the study of the Hilbert scheme of points in the plane. The latter relationship is surely closely related to our work, since this Hilbert scheme appears also as a fundamental object in the computation of the Donaldson-Thomas invariants in ~\cite{MNOP1, MNOP2}.  Indeed, it would not entirely shock us to learn that our result is known to geometers; however, we were not able to find a clear statement or a proof in the literature. Moreover, as both sides of Theorem~\ref{mainresult} may be defined completely combinatorially, one would like a combinatorial proof.  We provide such a proof here.

In Section~\ref{backgroundinfoandnotation}, we provide background information on partitions and hook lengths as well as the necessary components of the work of MNOP in \cite{MNOP1} and \cite{MNOP2} and full definition of our ``swap'' operation, the equivariant vertex measure, and the Jack Plancherel measure. Our main results appear in Section~\ref{mainresults}, in which we provide ratios of measures for growing a partition by one box in both the Jack Plancherel measure and the equivariant edge measure. We then show that the ratios of measures are equal. Routine induction reveals that the equivariant edge measure is the Jack Plancherel measure up to a sign.

\section{Background Information and Notation} \label{backgroundinfoandnotation}

This section explains necessary background information and set up to introduce the main results in the next section. We follow the exceptional exposition of Stanley from \cite{RS}; more information about the Jack symmetric functions can be found in \cite[Chapter 10]{M}.

\subsection{Partitions and Hook Lengths} \label{backgroundpartitionsandhooklengths}

A \textit{partition} $\lambda$ is a vector $(\lambda_1, \lambda_2, \lambda_3, \dots)$ with (not necessarily strictly) decreasing natural number entries only finitely many of which are nonzero. The \textit{size} of $\lambda$ is 
\begin{align}
    |\lambda| = \lambda_1 + \lambda_2 + \cdots . \notag
\end{align}
If $|\lambda| = n$ then we say that $\lambda$ \textit{partitions} $n$, denoted $\lambda \vdash n$. Each of the $\lambda_i$ is called a \textit{part} of $\lambda$. The number of nonzero parts of $\lambda$ is said to be the \textit{length} of $\lambda$, denoted $\text{len}(\lambda)$. In this paper, we identify a partition $\lambda$ with its \textit{Young} (or \textit{Ferrers}) \textit{diagram} in English notation with zero-indexed matrix coordinates, $(\text{row}, \text{column})$:
\begin{align}
    \lambda = \{ (i,j) \; | \; 0 \leq i \leq \text{len}(\lambda)-1, \; 0 \leq j \leq \lambda_i-1 \}. \notag
\end{align}
The Young diagram for the partition $\lambda = (3,3,1)$ is shown in Figure \ref{fig:exampleydiagram}.

\begin{figure}[ht]
    \begin{center}
            \begin{tikzpicture}[scale=0.5]
                \draw (0,0) -- (0,3);
                \draw (1,0) -- (1,3);
                \draw (2,1) -- (2,3);
                \draw (3,1) -- (3,3);
                \draw (0,0) -- (1,0);
                \draw (0,1) -- (3,1);
                \draw (0,2) -- (3,2);
                \draw (0,3) -- (3,3);
            \end{tikzpicture}
    \end{center}
    \caption{The diagram for $\lambda = (3,3,1)$, whose cells are: \{(0,0), (0,1), (0,2), (1,0), (1,1), (1,2), (2,0)\}.}
    \label{fig:exampleydiagram}
\end{figure}
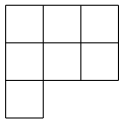

Given two partitions $\lambda$ and $\mu$, we write $\mu\subseteq\lambda$ if the $\mu_i \leq \lambda_i$ for all $i$. In other words, $\mu\subseteq\lambda$ if the diagram of $\mu$ fits inside the diagram of $\lambda$ when they are overlaid with their upper left cells aligned. The partially ordered set defined by the relation $\subseteq$ is called \textit{Young's Lattice}. The \textit{conjugate} $\lambda'$ of a partition $\lambda$ is the partition with diagram 
\begin{align}
    \lambda = \{ (j,i) \; | \; 0 \leq i \leq \text{len}(\lambda)-1, 0 \leq j \leq \lambda_i-1 \}, \notag
\end{align}
i.e., $\lambda'$ is the partition obtained by reflecting $\lambda$ across the northwest to southeast diagonal. Given a cell $(i,j)$ in $\lambda$, the \textit{arm length} and \textit{leg length} of $(i,j)$ are 
\begin{align}
    a_\lambda((i,j)) = \lambda_i - j, \qquad \ell_\lambda((i,j)) = \lambda'_j - i \notag
\end{align}
respectively.

Suppose that the cell $(i,j)$ is in $\lambda$. Then the \textit{hook length} of a cell $\square\in\lambda$ is 
\begin{align}
    h_\lambda(\square) = a_\lambda(\square)+\ell_\lambda(\square)+1. \notag
\end{align}
The partition $\lambda$ may be suppressed from the notation for arm length, leg length, and hook length if it is clear from context.
For a partition $\lambda$ define the \textit{upper} and \textit{lower hook lengths} as
\begin{align}
    h^*_\lambda(\square) &= u \cdot \ell(\square)+v(a(\square)+1) \notag \\
    h_*^\lambda(\square) &= u(1+ \ell(\square))+v \cdot a(\square) \notag
\end{align}
respectively. Note that both the upper and lower hook lengths specialize to the standard hook length on Young diagrams when $u$ and $v$ are both $1$. Indeed, this can be viewed as a homogenization of the Jack hook lengths described in \cite{RS}.  In the notation of~\cite{RS}, the Jack parameter is $\alpha$, which for us is equal to $v/u$.

\begin{figure}[ht]
    \centering
        \begin{tikzpicture}[scale=1.5]
            \draw (0,0) -- (0,2);
            \draw (1,0) -- (1,2);
            \draw (2,0) -- (2,2);
            \draw (3,1) -- (3,2);
            \draw (0,0) -- (2,0);
            \draw (0,1) -- (3,1);
            \draw (0,2) -- (3,2);
            \node at (0.5,0.6) {$2v$};
            \node at (0.5,0.4) {$v+u$};
            \node at (1.5,0.6) {$v$};
            \node at (1.5,0.4) {$u$};
            \node at (0.5,1.6) {$3v+u$};
            \node at (0.5,1.4) {$2v+2u$};
            \node at (1.5,1.6) {$2v+u$};
            \node at (1.5,1.4) {$v+2u$};
            \node at (2.5,1.6) {$v$};
            \node at (2.5,1.4) {$u$};
        \end{tikzpicture}
    \caption{The upper and lower hook lengths of each box in $\mu = (3,2)$.}
    \label{fig:exampleupperlowerhooklength}
\end{figure}

The Jack Plancherel measure $\wjack(\mu)$ (as defined in \eqref{eq:jackplancherel}) of the partition $\mu$ in Figure \ref{fig:exampleupperlowerhooklength} is the product of the inverses of each of the expressions printed in the cells of $\mu$. 

\subsection{Corners}

We say a cell $(i,j)$ in a partition $\lambda$ is an \textit{inside corner} of $\lambda$ if neither $(i+1,j)$ nor $(i,j+1)$ is in $\lambda$. We also add two artificial inside corners with coordinates $(-1,\infty)$ and $(\infty,-1)$. Similarly, we say a cell $(i,j)\not\in\lambda$ is an \textit{outside corner} of $\lambda$ if any of the following hold:
\begin{itemize}
    \item both of $(i-1,j)$ and $(i,j-1)$ are in $\lambda$
    \item $(i,j) = (1,\lambda_1+1)$
    \item $(i,j) = (\lambda'_1+1,1)$.
\end{itemize}

Label the inside corners of $\lambda$ from the bottom left to the top right as
\begin{align}\label{corners}
    (\rho_0, \gamma_0), (\rho_1, \gamma_1), (\rho_2, \gamma_2), \dots, (\rho_m, \gamma_m), (\rho_{m+1}, \gamma_{m+1})
\end{align}
where $\rho_0 = \gamma_{m+1} = \infty$ and $\rho_{m+1} = \gamma_0 = -1$. This means the outside corners are
\begin{align}
    (\rho_1 + 1, 0), (\rho_2 + 1, \gamma_1 + 1), (\rho_3 + 1, \gamma_2 + 1), \dots, (0, \gamma_m + 1). \notag
\end{align}

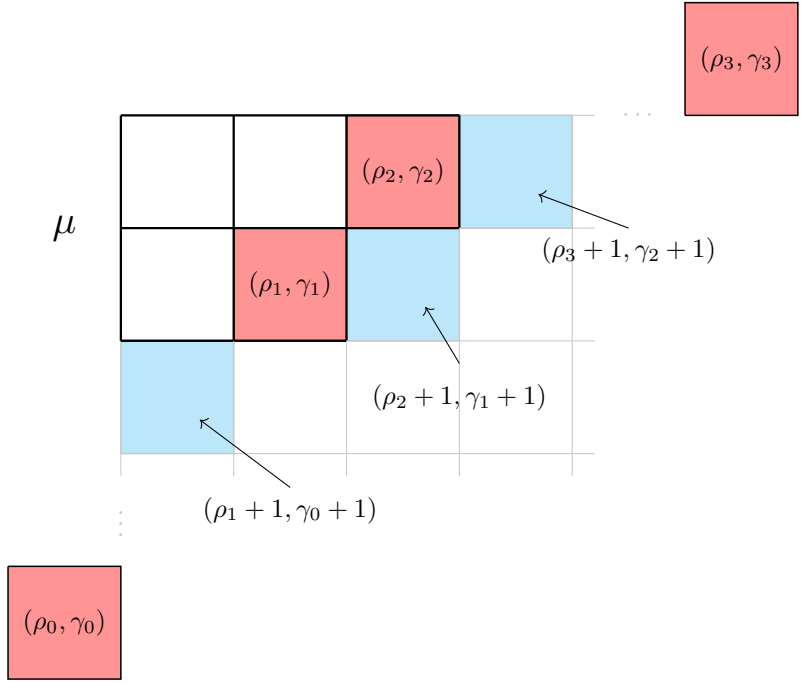
\begin{figure}[ht]
    \centering
        \begin{tikzpicture}[scale=1.5]
            \node[scale = 1.5] at (-0.5,1) {$\mu$};
            \draw [color = gray!42] (0,0) -- (0,-1.2);
            \node [color = gray!42] at (0,-1.55) {$\vdots$};
            \draw [color = gray!42] (0,2) -- (4.2,2);
            \node [color = gray!42] at (4.6, 2) {$\cdots$};
            \draw [line width=0.2mm,fill=red!42, ] (1,0) -- (1,1) -- (2,1) -- (2,0);
            \draw [line width=0.2mm,fill=red!42, ] (2,1) -- (2,2) -- (3,2) -- (3,1);
            \draw [line width=0.2mm,fill=red!42, ] (-1,-2) -- (0,-2) -- (0,-3) -- (-1,-3) -- (-1,-2);
            \draw [line width=0.2mm,fill=red!42, ] (5,2) -- (5,3) -- (6,3) -- (6,2) -- (5,2);
            \draw [line width=0.2mm, color = gray!42, fill=cyan!24, ] (0,-1) -- (0,0) -- (1,0) -- (1,-1) -- (0,-1);
            \draw [line width=0.2mm, color = gray!42, fill=cyan!24, ] (2,0) -- (2,1) -- (3,1) -- (3,0) -- (2,0);
            \draw [line width=0.2mm, color = gray!42, fill=cyan!24, ] (3,1) -- (3,2) -- (4,2) -- (4,1) -- (3,1);
            \node at (1.5,0.5) {$(\rho_1,\gamma_1)$};
            \node at (2.5,1.5) {$(\rho_2,\gamma_2)$};
            \node at (-0.5,-2.5) {$(\rho_0,\gamma_0)$};
            \node at (5.5,2.5) {$(\rho_3,\gamma_3)$};
            \draw [line width=0.3mm] (0,0) -- (2,0);
            \draw [line width=0.3mm] (0,0) -- (0,2);
            \draw [line width=0.3mm] (1,0) -- (1,2);
            \draw [line width=0.3mm] (2,0) -- (2,2);
            \draw [line width=0.3mm] (3,1) -- (3,2);
            \draw [line width=0.3mm] (0,0) -- (2,0);
            \draw [line width=0.3mm] (0,1) -- (3,1);
            \draw [line width=0.3mm] (0,2) -- (3,2);
            \draw [color = gray!42] (3,1) -- (4.2,1);
            \draw [color = gray!42] (2,0) -- (4.2,0);
            \draw [color = gray!42] (1,-1) -- (4.2,-1);
            \draw [color = gray!42] (1,0) -- (1,-1.2);
            \draw [color = gray!42] (2,0) -- (2,-1.2);
            \draw [color = gray!42] (3,1) -- (3,-1.2);
            \draw [color = gray!42] (4,2) -- (4,-1.2);
            \draw[ ->] (1.5,-1.3) -- (0.7,-0.7);
            \draw[ ->] (3,-0.2) -- (2.7,0.3);
            \draw[ ->] (4.5,1) -- (3.7,1.3);
            \node at (1.5,-1.5) {$(\rho_1 + 1,\gamma_0 + 1)$};
            \node at (3,-0.5) {$(\rho_2 + 1,\gamma_1 + 1)$};
            \node at (4.5,0.8) {$(\rho_3 + 1,\gamma_2 + 1)$};
        \end{tikzpicture}
    \caption{The inside (red/darker) and outside (blue/lighter) corners of $\mu = (3,2)$. Here, $(\rho_0, \rho_1, \rho_2, \rho_3) = (\infty, 1,0,-1)$ and $(\gamma_0, \gamma_1, \gamma_2, \gamma_3) = (-1,1,2,\infty)$.}
    \label{fig:examplecorners}
\end{figure}

\subsection{Equivariant Edge Measure Preliminaries}\label{MNOPpreliminaries}

We now turn to a discussion of the quantities $Q(\lambda)$, $\overline{Q}(\lambda)$, $F(\lambda)$, and $\wmnop(\lambda)$ defined in Section~\ref{sec:intro_dt}.  Where unambiguous, we write $Q=Q(\lambda)$ and so forth. The following two lemmas concern a generating function which we call the ``corner polynomial'' and a useful property of the swap operation.

\begin{lemma}\label{cornerpolynomial}
    The ``corner polynomial'' for a partition $\lambda$ is
    \begin{align}
        C = C(\lambda) := Q(1-r)(1-s) &= 1 + \sum_{(i,j) \text{ (true) inside corner of } \lambda} r^{i+1}s^{j+1} - \sum_{(i,j) \text{ outside corner of } \lambda} r^{i}s^{j}  \notag \\
        &= 1 + \sum_{k=1}^m r^{\rho_k+1}s^{\gamma_k+1} - \sum_{k=1}^{m+1} r^{\rho_k+1}s^{\gamma_{k-1}+1}. 
    \end{align}
\end{lemma}

\begin{proof}
    Observe that
    \begin{align}\label{eq:cornerexpansion}
        Q(1-r)(1-s) = Q - Qr - Qs + Qrs.
    \end{align}
    We consider the contribution of each cell in the plane to $C$. There are five cases.
    \begin{enumerate}
        \item[\textbf{Case 1.}] The top left cell in any diagram $\lambda$ is only represented in the first term on the RHS of \eqref{eq:cornerexpansion} meaning that this cell will contribute $r^0s^0 = 1$ to $C$.
        \item[\textbf{Case 2.}] The remainder of cells inside the diagram $\lambda$ have a contribution of zero to $C$.
        \begin{itemize}
            \item If the cell is in the top row of $\lambda$, then is has a cell in $\lambda$ to its left, but not above it nor northwest of it. Hence it appears once in the first term and once in the third term of the RHS of \eqref{eq:cornerexpansion} which cancel making its total contribution 0.
            \item If the cell is in the far left column of $\lambda$, we can make an analogous argument.
            \item If the cell is otherwise in $\lambda$, it appears once in each of the four terms of the RHS of \eqref{eq:cornerexpansion} yielding a total contribution of 0.
        \end{itemize}
        \item[\textbf{Case 3.}] Note that cells $(i,j)$ which are outside corners of $\lambda$ either 
        \begin{itemize}
            \item have cells of $\lambda$ to the left, above, and northwest of themselves, hence appearing once in each of the last three terms of the RHS of \eqref{eq:cornerexpansion},
            \item have a cell of $\lambda$ to the left of themselves, hence appearing in the $-Qs$ term of \eqref{eq:cornerexpansion} only, or
            \item have a cell of $\lambda$ above themselves, hence appearing in the $-Qr$ term of \eqref{eq:cornerexpansion} only.
        \end{itemize} 
        In any case,  each outside corner of $\lambda$ contributes a total of $-r^is^j$ to $C$.
        \item[\textbf{Case 4.}] Given a (true, not artificial) inside corner $(i,j)$ in $\lambda$, the cell $(i+1,j+1)$ contributes $r^{i+1}s^{j+1}$ to $C$ since it appears once in the last term of the RHS of \eqref{eq:cornerexpansion} but not in any other terms. 
        \item[\textbf{Case 5.}] All as yet unmentioned cells are outside of $\lambda$ and do not have cells of $\lambda$ directly to the right, above, or northwest of themselves, hence have no effect on $C$.
    \end{enumerate}
\end{proof}

Define $\overline{C}$ similarly:
\begin{align}
    \overline{C} := \overline{Q}(1-r^{-1})(1-s^{-1}) = \frac{\overline{Q}(1-r)(1-s)}{rs}. \notag
\end{align}

\begin{figure}
    \centering
        \begin{tikzpicture}[scale=1]
            \draw [line width=0.2mm,fill=green!30, ] (0,0) -- (0,1) -- (1,1) -- (1,0);
            \draw [line width=0.2mm,fill=green!30, ] (1,0) -- (1,1) -- (2,1) -- (2,0);
            \draw [line width=0.2mm,fill=green!30, ] (0,1) -- (0,2) -- (1,2) -- (1,1);
            \draw [line width=0.2mm,fill=green!30, ] (1,1) -- (1,2) -- (2,2) -- (2,1);
            \draw [line width=0.2mm,fill=green!30, ] (2,1) -- (2,2) -- (3,2) -- (3,1);
            \draw (0,0) -- (0,2);
            \draw (1,0) -- (1,2);
            \draw (2,0) -- (2,2);
            \draw (3,1) -- (3,2);
            \draw (0,0) -- (2,0);
            \draw (0,1) -- (3,1);
            \draw (0,2) -- (3,2);
            \node at (0.5,1.5) {$1$};
            \node at (3.5,1.5) {$-1$};
            \node at (2.5,0.5) {$-1$};
            \node at (0.5,-0.5) {$-1$};
            \node at (2.5,-0.5) {$1$};
            \node at (3.5,0.5) {$1$};
            \draw[color = gray!42] (0,-2.2) -- (0,0);
            \draw[color = gray!42] (1,-2.2) -- (1,0);
            \draw[color = gray!42] (2,-2.2) -- (2,0);
            \draw[color = gray!42] (3,-2.2) -- (3,1);
            \draw[color = gray!42] (4,-2.2) -- (4,2);
            \draw[color = gray!42] (5,-2.2) -- (5,2);
            \draw[color = gray!42] (3,2) -- (5.2,2);
            \draw[color = gray!42] (3,1) -- (5.2,1);
            \draw[color = gray!42] (2,0) -- (5.2,0);
            \draw[color = gray!42] (0,-1) -- (5.2,-1);
            \draw[color = gray!42] (0,-2) -- (5.2,-2);
            \node at (-0.5,1) {$\lambda$};
        \end{tikzpicture}
    \caption{Inside every cell in $\lambda = (3,2)$ is the coefficient of its contribution to $C$. Empty cells contribute nothing to $C$. For example, the cell $(1,2)$ contributes $-1 \cdot r^1s^2$ to $C$.}
    \label{fig:examplecornerpolycontributions}
\end{figure}

\begin{lemma} \label{quotientofswapsisswapofdifference}
    For Laurent polynomials $F$ and $G$ in the variables $r$ and $s$ with no constant terms, the quotient of swaps of $F$ and $G$ is the swap of the difference of $F$ and $G$. In other words,
    \begin{align}
        \text{swap}(F - G) = \frac{\text{swap}(F)}{\text{swap}(G)}.
    \end{align}
\end{lemma}
\begin{proof}
    Suppose
    \begin{align}
        F &= \sum_{i,j} c_{i,j} r^i s^j, &
        G &= \sum_{i,j} d_{i,j} r^i s^j. \notag
    \end{align}
    Applying the swap operation yields
    \begin{align}
        \text{swap}(F) &= \prod_{i,j} (iu - jv)^{c_{i,j}}, \notag&
        \text{swap}(G) &= \prod_{i,j} (iu - jv)^{d_{i,j}}. \notag
    \end{align}
    Note also that
    \begin{align}
        F - G = \sum_{i,j} (c_{i,j} - d_{i,j}) r^i s^j \notag
    &&\text{and}&&
        \text{swap}(F - G) &= \prod_{i,j} (iu - jv)^{c_{i,j} - d_{i,j}}. \notag
    \end{align}
    Taking the quotient of $\text{swap}(F)$ and $\text{swap}(G)$, we have
    \begin{align}
        \frac{\text{swap}(F)}{\text{swap}(G)} = \frac{\prod_{i,j} (iu - jv)^{c_{i,j}}}{\prod_{i,j} (iu - jv)^{d_{i,j}}} = \prod_{i,j} (iu - jv)^{c_{i,j} - d_{i,j}} = \text{swap}(F - G). \notag
    \end{align}
\end{proof}

\section{Main Results}\label{mainresults}

In this section, we compute both the ratio of Jack Plancherel measures and the ratio of equivariant edge measures of two partitions that differ by exactly one corner. Once this is accomplished, we compare these ratios. Then Theorem \ref{mainresult} follows by induction on Young's lattice.

Given a partition $\lambda$ and a distinguished inside corner of $\lambda$, $(\rho_\ell, \gamma_\ell)$, define a new partition $\mu$ by removing $(\rho_\ell, \gamma_\ell)$ from $\lambda$. Note that $\mu\subseteq\lambda$ and $|\lambda| = |\mu| + 1$. Figure \ref{fig:examplegrowingbyonebox} shows an example of such a pair of Young diagrams.

\begin{figure}[ht]
    \centering
    \begin{subfigure}[b]{0.45\textwidth}
        \centering
        \begin{tikzpicture}[scale=1.3]
           \draw (0,-1) -- (0,2);
           \draw (1,-1) -- (1,2);
           \draw (2,0) -- (2,2);
           \draw (3,1) -- (3,2);
           \draw (0,-1) -- (1,-1);
           \draw (0,0) -- (2,0);
           \draw (0,1) -- (3,1);
           \draw (0,2) -- (3,2);
           \node at (1.5,0.5) {$(\rho_\ell, \gamma_\ell)$};
           %\node at (2.5, -0.5) {\Large$\lambda$};
        \end{tikzpicture}
        \caption{$\lambda = (3,2,1)$}
        \label{fig:examplegrowbyasquarelambda}   
    \end{subfigure}
    %\hspace{1in}
    \begin{subfigure}[b]{0.45\textwidth}
        \centering
        \begin{tikzpicture}[scale=1.3]
            \draw (0,-1) -- (0,2);
            \draw (1,-1) -- (1,2);
            \draw (2,1) -- (2,2);
            \draw (3,1) -- (3,2);
            \draw (0,-1) -- (1,-1);
            \draw (0,0) -- (1,0);
            \draw (0,1) -- (3,1);
            \draw (0,2) -- (3,2);
            %\node at (2.5, -0.5) {\Large$\mu$};
        \end{tikzpicture}
        \caption{$\mu = (3,1,1)$}
        \label{fig:examplegrowbyasquaremu}   
    \end{subfigure}
    \caption{The partition $\mu$ is obtained by removing the corner $(\rho_\ell, \gamma_\ell)$ from $\lambda$.}
    \label{fig:examplegrowingbyonebox}
\end{figure}
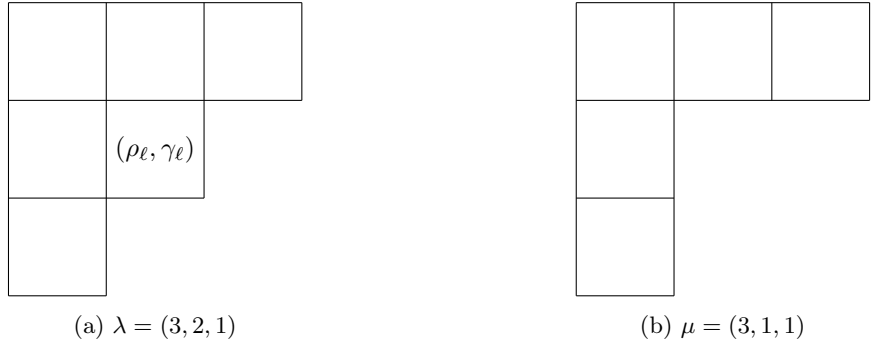

\subsection{Ratio of Jack Plancherel Measures}

Recall that we defined the Jack Plancherel measure as
\begin{align}
    \wjack(\lambda) = \frac{1}{\prod_{\square\in\lambda}h^*_\lambda(\square)h_*^\lambda(\square)} \notag
\end{align}
in Section \ref{introductionmotivation}. We now compute the ratio of Jack Plancherel measures of two partitions $\lambda$ and $\mu$ related as described above.

\begin{proposition} \label{JPWR}
    Let $\mu$ and $\lambda$ be partitions such that $|\mu| + 1 = |\lambda|$ and $\mu\subseteq\lambda$ and let $m$ be the number of non-artificial inside corners of $\lambda$. Then
    \begin{align}
        \frac{\wjack(\lambda)}{\wjack(\mu)} = \prod_{k=1}^{m} A(k) \notag
    \end{align}
    where 
    \begin{align}\label{eq:JPWR}
        A(k) = 
        \begin{cases}
            \frac{((\rho_\ell - \rho_k - 1)u - (\gamma_\ell-\gamma_k - 1)v)  \cdot ((\rho_\ell - \rho_k)u - (\gamma_\ell-\gamma_k)v)} {((\rho_\ell - \rho_k - 1)u - (\gamma_\ell-\gamma_{k-1} - 1)v) \cdot ((\rho_\ell - \rho_k)u - (\gamma_\ell-\gamma_{k-1})v)} & \text{if } k < \ell \\
            \\
            \frac{((\rho_\ell - \rho_k)u - (\gamma_\ell-\gamma_k)v) \cdot ((\rho_\ell - \rho_k - 1)u - (\gamma_\ell-\gamma_k - 1)v)} {((\rho_\ell - \rho_{k+1})u - (\gamma_\ell - \gamma_k)v) \cdot ((\rho_\ell - \rho_{k+1} - 1)u - (\gamma_\ell-\gamma_k - 1)v)} & \text{if } k > \ell \\
            \\
            \frac{1}{(\rho_\ell - \rho_{\ell+1})u \cdot ((\rho_\ell - \rho_{\ell+1} - 1)u + v)} \cdot \frac{uv}{(u + (\gamma_{\ell} - \gamma_{\ell-1} - 1)v) \cdot ((\gamma_{\ell} - \gamma_{\ell-1})v)} & \text{if } k = \ell.
        \end{cases}
    \end{align}
\end{proposition}

\begin{proof}
    Label the corners of the partition $\lambda$ as in \eqref{corners}. We have 
    \begin{align}\label{eq:firststepjackratioproof}
        \frac{\wjack(\lambda)}{\wjack(\mu)} = \frac{\prod_{\square\in\mu} h^*_\mu(\square)h_*^\mu(\square)}{\prod_{\square\in\lambda} h^*_\lambda(\square)h_*^\lambda(\square)} = \frac{1}{uv} \cdot \prod_{\square\in\mu} \frac{h^*_\mu(\square)h_*^\mu(\square)}{h^*_\lambda(\square)h_*^\lambda(\square)}.
    \end{align}
    Since $\lambda$ and $\mu$ differ by exactly one cell, namely $(\rho_\ell, \gamma_\ell)\in\lambda$, we note that $h^*_\lambda(\square) = h^*_\mu(\square)$ for all cells in $\mu$ except for those in the row or column of $(\rho_\ell,\gamma_\ell)$. The same can be said for the lower hook lengths. Hence significant cancellation occurs in \eqref{eq:firststepjackratioproof}.
    
    Call the set of cells that share a row with $(\rho_\ell,\gamma_\ell)$ (i.e. directly to the left of $(\rho_\ell,\gamma_\ell)$) $\row((\rho_\ell,\gamma_\ell))$ and the set of cells sharing a column with $(\rho_\ell,\gamma_\ell)$ (i.e. directly above $(\rho_\ell,\gamma_\ell)$) $\col((\rho_\ell,\gamma_\ell))$.
    
    \begin{figure}[ht] 
        \begin{center}
            \begin{tikzpicture}[scale=1.5]
                \draw (0,-1) -- (0,4);
                \draw (1,-1) -- (1,4);
                \draw (2,0) -- (2,4);
                \draw (3,1) -- (3,4);
                \draw (-1,3) -- (3,3);
                \draw (-1,4) -- (3,4);
                \draw (-1,-1) -- (1,-1);
                \draw (-1,0) -- (2,0);
                \draw (-1,1) -- (3,1);
                \draw (-1,2) -- (3,2);
                \draw (-1,-1) -- (-1,4);
                \node at (2.5, -0.5) {\Large$\lambda$};
                \draw [line width=0mm,fill=red!42, ] (1,4) -- (1,1) -- (2,1) -- (2,4);
                \node [rotate = 270] at (1.5,2.5) {$\col((\rho_\ell,\gamma_\ell))$};
                \node at (1.5,0.5) {$(\rho_\ell, \gamma_\ell)$};
                \draw [line width=0mm,fill=cyan!24, ] (-1,1) -- (-1,0) -- (1,0) -- (1,1);
                 \node at (0,0.5) {$\row((\rho_\ell,\gamma_\ell))$};
                 \node at (-1.5,1.5) {$\cdots$};
                 \node at (1,4.5) {$\vdots$};
            \end{tikzpicture}
            \caption{$\row((\rho_\ell,\gamma_\ell))$ and $\col((\rho_\ell,\gamma_\ell))$.}
            \label{fig:examplerowandcol}
        \end{center}
    \end{figure}
    
    Now 
    \begin{align}
        \frac{\wjack(\lambda)}{\wjack(\mu)} = \frac{1}{uv} \cdot \prod_{\square\in\row((\rho_\ell,\gamma_\ell))} \frac{h^*_\mu(\square)h_*^\mu(\square)}{h^*_\lambda(\square)h_*^\lambda(\square)} \cdot \prod_{\square\in\col((\rho_\ell,\gamma_\ell))} \frac{h^*_\mu(\square)h_*^\mu(\square)}{h^*_\lambda(\square)h_*^\lambda(\square)}. \notag
    \end{align}
    We split these products to view the expression as a product over corners of $\lambda$.
    \begin{align}\label{eq:JPWRproofstep}
        \frac{\wjack(\lambda)}{\wjack(\mu)} = \frac{1}{uv} \cdot \prod_{k = 1}^{\ell - 1} \left( \prod_{i = \gamma_{k-1} + 1}^{\gamma_k} \frac{h^*_\mu((\rho_\ell, i))h_*^\mu((\rho_\ell, i))}{h^*_\lambda((\rho_\ell, i))h_*^\lambda((\rho_\ell, i))} \right)
        \cdot \prod_{k = \ell + 1}^{m} \left( \prod^{\rho_{k}}_{i = \rho_{k+1} + 1} \frac{h^*_\mu((i, \gamma_\ell))h_*^\mu((i, \gamma_\ell))}{h^*_\lambda((i, \gamma_\ell))h_*^\lambda((i, \gamma_\ell))} \right) \notag\\
        \cdot \prod_{i = \gamma_{\ell - 1} + 1}^{\gamma_\ell - 1} \frac{h^*_\mu((\rho_\ell, i))h_*^\mu((\rho_\ell, i))}{h^*_\lambda((\rho_\ell, i))h_*^\lambda((\rho_\ell, i))}
        \cdot \prod^{\rho_\ell - 1}_{i = \rho_{\ell + 1} + 1} \frac{h^*_\mu((i, \gamma_\ell))h_*^\mu((i, \gamma_\ell))}{h^*_\lambda((i, \gamma_\ell))h_*^\lambda((i, \gamma_\ell))}
    \end{align}
    How we proceed depends on if the last two products of \eqref{eq:JPWRproofstep} are empty or nonempty. We complete this argument via cases.
    \begin{enumerate}
        \item[\textbf{Case 1.}] If the final two products are nonempty each inside product in the first line and both of the last two products telescope, leaving
        \begin{align}
            \frac{\wjack(\lambda)}{\wjack(\mu)} = \frac{1}{uv} \cdot \prod_{k = 1}^{\ell - 1} \frac{h^*_\mu((\rho_\ell, \gamma_k))h_*^\mu((\rho_\ell, \gamma_k))}{h^*_\lambda((\rho_\ell, \gamma_{k-1}+1))h_*^\lambda((\rho_\ell, \gamma_{k-1}+1))}
            \cdot \prod_{k = \ell + 1}^{m} \frac{h^*_\mu(\rho_k,\gamma_\ell)h_*^\mu(\rho_k,\gamma_\ell)}{h^*_\lambda(\rho_{k+1}+1,\gamma_\ell)h_*^\lambda(\rho_{k+1}+1,\gamma_\ell)}  \notag \\
            \cdot \frac{h^*_\mu((\rho_\ell, \gamma_\ell - 1))h_*^\mu((\rho_\ell, \gamma_\ell - 1))}{h^*_\lambda((\rho_\ell, \gamma_{\ell - 1} + 1))h_*^\lambda((\rho_\ell, \gamma_{\ell - 1} + 1))}
            \cdot \frac{h^*_\mu((\rho_\ell - 1, \gamma_\ell))h_*^\mu((\rho_\ell - 1, \gamma_\ell))}{h^*_\lambda((\rho_{\ell + 1} + 1, \gamma_\ell))h_*^\lambda((\rho_{\ell + 1} + 1, \gamma_\ell))}.  \notag
        \end{align}
        Rewriting without the use of hook lengths, we have
        \begin{align}
            \frac{\wjack(\lambda)}{\wjack(\mu)} = \prod_{k = 1}^{\ell - 1} \frac{((\rho_k - \rho_\ell)u + (\gamma_\ell - \gamma_k)v) ((\rho_k - \rho_\ell + 1)u + (\gamma_\ell - \gamma_k - 1)v)}{ ((\rho_k - \rho_\ell)u + (\gamma_\ell - \gamma_{k-1})v) ((\rho_k - \rho_\ell + 1)u + (\gamma_\ell - \gamma_{k-1} - 1)v)}   \notag \\
            \cdot \prod_{k = \ell + 1}^{m} \frac{((\rho_\ell - \rho_k - 1)u+(\gamma_k - \gamma_\ell + 1)v) ((\rho_\ell - \rho_k)u +(\gamma_k - \gamma_\ell)v)}{((\rho_\ell - \rho_{k+1} - 1)u + (\gamma_k - \gamma_\ell + 1)v) ((\rho_\ell - \rho_{k+1})u + (\gamma_k - \gamma_\ell)v)}  \notag \\
            \cdot \frac{1}{uv} \cdot \frac{uv}{((\gamma_\ell - \gamma_{\ell - 1})v) (u + (\gamma_\ell - \gamma_{\ell - 1} - 1)v)} \cdot \frac{uv}{((\rho_\ell - \rho_{\ell+1} - 1)u + v) ((\rho_\ell - \rho_{\ell+1})u)}.  \notag
        \end{align}
        Simplification gives the desired result.

        \item[\textbf{Case 2.}] If the last two products are empty we can omit them. Figure \ref{fig:examplerowandcol} depicts a situation in which this occurs. Note that in this case, $\rho_{\ell+1} = \rho_\ell - 1$ and $\gamma_{\ell-1} = \gamma_\ell - 1$. With this in mind, the $k=\ell$ case of the ratio of Jack Plancherel measures \eqref{eq:JPWR} reduces to $\frac{1}{uv}$.
        
        Picking up where we left off, as in Case 1, each inside product in \eqref{eq:JPWRproofstep} telescopes. We are left with
        \begin{align}
            \frac{\wjack(\lambda)}{\wjack(\mu)} = \frac{1}{uv} \cdot \prod_{k = 1}^{\ell - 1} \frac{h^*_\mu((\rho_\ell, \gamma_k))h_*^\mu((\rho_\ell, \gamma_k))}{h^*_\lambda((\rho_\ell, \gamma_{k-1}+1))h_*^\lambda((\rho_\ell, \gamma_{k-1}+1))} 
            \cdot \prod_{k = \ell + 1}^{m} \frac{h^*_\mu(\rho_k,\gamma_\ell)h_*^\mu(\rho_k,\gamma_\ell)}{h^*_\lambda(\rho_{k+1}+1,\gamma_\ell)h_*^\lambda(\rho_{k+1}+1,\gamma_\ell)}.  \notag 
        \end{align}
        Rewriting without the use of hook lengths, we have
        \begin{align}
            \frac{\wjack(\lambda)}{\wjack(\mu)} &= \frac{1}{uv} \cdot \prod_{k = 1}^{\ell - 1} \frac{((\rho_k - \rho_\ell)u + (\gamma_\ell - \gamma_k)v) ((\rho_k - \rho_\ell + 1)u + (\gamma_\ell - \gamma_k - 1)v)}{ ((\rho_k - \rho_\ell)u + (\gamma_\ell - \gamma_{k-1})v) ((\rho_k - \rho_\ell + 1)u + (\gamma_\ell - \gamma_{k-1} - 1)v)}  \notag \\
            &\hspace{0.5in}\cdot \prod_{k = \ell + 1}^{m} \frac{((\rho_\ell - \rho_k - 1)u+(\gamma_k - \gamma_\ell + 1)v) ((\rho_\ell - \rho_k)u +(\gamma_k - \gamma_\ell)v)}{((\rho_\ell - \rho_{k+1} - 1)u + (\gamma_k - \gamma_\ell + 1)v) ((\rho_\ell - \rho_{k+1})u + (\gamma_k - \gamma_\ell)v)}.  \notag
        \end{align} 
        As a product over corners of $\lambda$ we have
        \begin{align}
            \frac{\wjack(\lambda)}{\wjack(\mu)} = \prod_{k=1}^{\# \text{ of corners of } \lambda} A(k)  \notag
        \end{align}
        where 
        \begin{align}
            A(k) = 
            \begin{cases}
                \frac{((\rho_\ell - \rho_k - 1)u - (\gamma_\ell-\gamma_k - 1)v)  \cdot ((\rho_\ell - \rho_k)u - (\gamma_\ell-\gamma_k)v)} {((\rho_\ell - \rho_k - 1)u - (\gamma_\ell-\gamma_{k-1} - 1)v) \cdot ((\rho_\ell - \rho_k)u - (\gamma_\ell-\gamma_{k-1})v)} & \text{if } k < \ell \\
                \\
                \frac{((\rho_\ell - \rho_k)u - (\gamma_\ell-\gamma_k)v) \cdot ((\rho_\ell - \rho_k - 1)u - (\gamma_\ell-\gamma_k - 1)v)} {((\rho_\ell - \rho_{k+1})u - (\gamma_\ell - \gamma_k)v) \cdot ((\rho_\ell - \rho_{k+1} - 1)u - (\gamma_\ell-\gamma_k - 1)v)} & \text{if } k > \ell \\
                \\
                \frac{1}{uv} & \text{if } k = \ell.  \notag
            \end{cases}
        \end{align}
        as claimed.
    \end{enumerate}  
\end{proof}

\subsection{Ratio of Equivariant Edge Measures}

We now compute the ratio of the equivariant edge measures of two partitions that differ by a single box. 

\begin{proposition} \label{MNOPWR}
    Let $\mu$ and $\lambda$ be partitions such that $|\mu| + 1 = |\lambda|$ and $\mu\subseteq\lambda$. Then
    \begin{align}
        \frac{\wmnop(\lambda)}{\wmnop(\mu)} = \prod_{k=1}^m B(k)  \notag
    \end{align}
    where
    \begin{align}
        B(k) = 
        \begin{cases}
            \frac{((\rho_\ell-\rho_k)u - (\gamma_\ell-\gamma_k)v) ((\rho_\ell - \rho_k - 1)u - (\gamma_\ell - \gamma_k - 1)v)} {((\rho_\ell-\rho_k)u - (\gamma_\ell -\gamma_{k-1})v) ((\rho_\ell - \rho_{k+1} - 1)u - (\gamma_\ell - \gamma_k - 1)v)} & \text{ if } k<\ell \\
            \\
            \frac{((\rho_\ell-\rho_k)u - (\gamma_\ell-\gamma_k)v) ((\rho_\ell - \rho_k - 1)u - (\gamma_\ell - \gamma_k - 1)v)} {((\rho_\ell-\rho_k)u - (\gamma_\ell -\gamma_{k-1})v) ((\rho_\ell - \rho_{k+1} - 1)u - (\gamma_\ell - \gamma_k - 1)v)} & \text{ if } k>\ell \\
            \\
            \frac{-uv}{(\gamma_{\ell} - \gamma_{\ell-1})v \cdot ((\rho_\ell - \rho_{\ell+1} - 1)u + v)} \cdot \frac{1}{((\rho_\ell)u - (\gamma_\ell - \gamma_m)v)((\rho_\ell - \rho_1 - 1)u - (\gamma_\ell - 1)v)} & \text{ if } k=\ell.
        \end{cases}
    \end{align}
\end{proposition}

\begin{proof}
    Once again, label the corners of the partition $\lambda$ as in \eqref{corners}. By Lemma \ref{quotientofswapsisswapofdifference} it suffices to compute $\text{swap}(F(\lambda)-F(\mu))$. Note that
    \begin{align}
        F(\lambda) - F(\mu) = - Q(\lambda) - \frac{\overline{Q}(\lambda)}{rs} + \frac{Q(\lambda)\overline{Q}(\lambda)(1-r)(1-s)}{rs} - \left( - Q(\mu) - \frac{\overline{Q}(\mu)}{rs} + \frac{Q(\mu)\overline{Q}(\mu)(1-r)(1-s)}{rs}\right). \notag
    \end{align}
    Since $\lambda$ and $\mu$ differ by exactly one cell $(\rho_\ell,\gamma_\ell)\in\lambda$, we have
    \begin{align}
        Q(\mu) = Q(\lambda) - r^{\rho_\ell}s^{\gamma_\ell}. \notag
    \end{align}
    Substituting,
    \begin{align}
        F(\lambda) - F(\mu)
        &= -r^{\rho_\ell}s^{\gamma_\ell} - r^{-\rho_\ell - 1}s^{-\gamma_\ell - 1} + r^{-\rho_\ell - 1}s^{-\gamma_\ell - 1}Q(\lambda)(1-r)(1-s) \notag \\
        &\hspace{2in}+ r^{\rho_\ell - 1}s^{\gamma_\ell - 1}\overline{Q}(\lambda)(1-r)(1-s) - \frac{(1-r)(1-s)}{rs}. \notag
    \end{align}
    In terms of corner polynomials, 
    \begin{align}
        F(\lambda) - F(\mu) = - r^{\rho_\ell}s^{\gamma_\ell} - r^{-\rho_\ell - 1}s^{-\gamma_\ell - 1} + r^{-\rho_\ell - 1}s^{-\gamma_\ell - 1}C(\lambda) + r^{\rho_\ell}s^{\gamma_\ell}\overline{C}(\lambda) - \frac{(1-r)(1-s)}{rs}. \notag
    \end{align}
    Simplifying using Lemma \ref{cornerpolynomial} we have
    \begin{align}
        F(\lambda) - F(\mu) &= \sum_{k=1}^{\ell-1} r^{\rho_k - \rho_\ell}s^{\gamma_k - \gamma_\ell}
        + \sum_{k=\ell + 1}^{m} r^{\rho_k - \rho_\ell}s^{\gamma_k - \gamma_\ell}
        - \sum_{k=1}^{m+1} r^{\rho_k - \rho_\ell} s^{\gamma_{k-1} - \gamma_\ell} \notag \\
        &+ \sum_{k=1}^{\ell-1} r^{\rho_\ell - \rho_k - 1}s^{\gamma_\ell - \gamma_k - 1}
        + \sum_{k=\ell + 1}^{m} r^{\rho_\ell - \rho_k - 1}s^{\gamma_\ell - \gamma_k - 1}
        - \sum_{k=0}^{m} r^{\rho_\ell - \rho_{k+1} - 1} s^{\gamma_\ell - \gamma_{k} -1}. \notag
    \end{align}
    We can now apply the swap operation:
    \begin{align}
        \text{swap}(F(\lambda) - F(\mu)) =& \prod_{k=1}^{\ell-1} ((\rho_k - \rho_\ell) u - (\gamma_k - \gamma_\ell) v) \cdot ((\rho_\ell - \rho_k - 1) u - (\gamma_\ell - \gamma_k - 1) v) \notag \\
        &\cdot \prod_{k=\ell + 1}^{m} ((\rho_k - \rho_\ell) u - (\gamma_k - \gamma_\ell) v) \cdot ((\rho_\ell - \rho_k - 1) u - (\gamma_\ell - \gamma_k - 1) v) \notag \\
        &\cdot \left( \prod_{k=1}^m ((\rho_k - \rho_\ell) u - (\gamma_{k-1} - \gamma_\ell) v) \cdot ((\rho_\ell - \rho_{k+1} - 1) u - (\gamma_\ell - \gamma_k - 1) v) \right)^{-1} \notag \\
        &\cdot -uv \cdot ( (\rho_{m+1} - \rho_\ell) u - (\gamma_m - \gamma_\ell) v )^{-1} \cdot ( (\rho_\ell - \rho_{1} -1) u - (\gamma_\ell - \gamma_0 - 1) v)^{-1}. \notag
    \end{align}
    Splitting this into the three cases $k<\ell$, $k>\ell$, $k=\ell$ yields the desired result.
\end{proof}

\subsection{Comparison of Ratios}

So far, we have expressed ratios of both the Jack Plancherel measure and the equivariant edge measure as products over the corners of $\lambda$. It remains to show that these ratios are equal, and then use that information inductively to show that $\wjack(\lambda) = -\wmnop(\lambda)$. 

\begin{theorem} \label{ratiotheorem}
    %The Jack Plancherel measure ratio and the equivariant edge measure ratio are the same. More precisely,
    We have
    \begin{align}
        \frac{\wjack(\lambda)}{\wjack(\mu)} = \frac{\wmnop(\lambda)}{\wmnop(\mu)}.
    \end{align}
\end{theorem}
\begin{proof}
    % It suffices to verify that 
    % \begin{align}
    %     \frac{\frac{\wjack(\lambda)}{\wjack(\mu)}}{\frac{\wmnop(\lambda)}{\wmnop(\mu)}} = 1. \notag
    % \end{align}
    We have
    \begin{align}
        \frac{\frac{\wjack(\lambda)}{\wjack(\mu)}}{\frac{\wmnop(\lambda)}{\wmnop(\mu)}} = \prod_{k = 1}^m \frac{A(k)}{B(k)} \notag
    \end{align}
    by Proposition \ref{JPWR} and Proposition \ref{MNOPWR}. We simplify this in three cases: $k<\ell$, $k>\ell$, and $k=\ell$. When $k<\ell$, all but one term of $A(k)$ agrees with $B(k)$. The remaining term in $A(k)$ is 
    \begin{align}
        T_A(k) := \frac{1}{(\rho_\ell-\rho_k - 1)u - (\gamma_\ell - \gamma_{k-1} - 1)v} \notag
    \end{align}
    and for $B(k)$ is 
    \begin{align}
        T_B(k) := \frac{1}{(\rho_\ell-\rho_{k+1} - 1)u - (\gamma_\ell - \gamma_{k} - 1)v}. \notag
    \end{align}
    Note that $T_B(k) = T_A(k+1)$. So via traditional cancellation and telescoping,
    \begin{align}
        \prod_{k = 1}^{\ell-1} \frac{A(k)}{B(k)} = \prod_{k = 1}^{\ell-1} \frac{T_A(k)}{T_B(k)} = \frac{T_A(1)}{T_B(\ell - 1)} = \frac{- u - (\gamma_\ell - \gamma_{\ell-1} - 1)v}{ (\rho_\ell - \rho_1 - 1)u - (\gamma_{\ell} - 1)v}. \notag
    \end{align}
    We proceed analogously with the $k>\ell$ case. Define 
    \begin{align*}
        T^A(k) &= \frac{1}{(\rho_\ell - \rho_{k+1})u - (\gamma_\ell - \gamma_k)v} &&
    \text{and}&
        T^B(k) &= \frac{1}{(\rho_\ell - \rho_{k})u - (\gamma_\ell - \gamma_{k-1})v}. 
    \end{align*}
    Then
    \begin{align}
        \prod_{k = \ell + 1}^{m} \frac{A(k)}{B(k)} = \prod_{k = \ell + 1}^{m} \frac{T^A(k)}{T^B(k)} = \frac{T^A(m)}{T^B(\ell + 1)} = \frac{(\rho_\ell - \rho_{\ell + 1})u}{(\gamma_\ell - \gamma_m)v - (\rho_\ell)u}. \notag
    \end{align}
    \begin{remark}
        \normalfont{When $\lambda$ has zero corners, i.e. when $m = 0$, the products above are empty. When $\lambda$ has one corner then one can verify that the two expressions above are both one.}
    \end{remark}
    Now
    \begin{align}
        \frac{\frac{\wjack(\lambda)}{\wjack(\mu)}}{\frac{\wmnop(\lambda)}{\wmnop(\mu)}} = \prod_{k = 1}^m \frac{A(k)}{B(k)} = \frac{(-1)((\gamma_\ell - \gamma_{\ell-1} - 1)v + u)}{(\gamma_{\ell} - 1)v - (\rho_\ell - \rho_1 - 1)u} \cdot \frac{(\rho_\ell - \rho_{\ell + 1})u}{(\gamma_\ell - \gamma_m)v - (\rho_\ell)u} \cdot \frac{A(\ell)}{B(\ell)}. \notag
    \end{align}
    We cancel, revealing 
    \begin{align}
        \frac{\frac{\wjack(\lambda)}{\wjack(\mu)}}{\frac{\wmnop(\lambda)}{\wmnop(\mu)}} = 1. \notag
    \end{align}
\end{proof}

\begin{proof}[Proof of Theorem \ref{mainresult}]
    We proceed by induction on the size of a partition using the partition $\nu$ with one box as our base case. Let $s$ be the unique cell in $\nu$. We have
    \begin{align}
        \wjack(\nu) = \frac{1}{h^*(s)h_*(s)} = \frac{1}{uv} \notag
    \end{align}
    and
    \begin{align}
        \wmnop(\nu) &= \swap\left(-Q(\nu) - \frac{\overline{Q}(\nu)}{rs} + \frac{Q(\nu)\overline{Q}(\nu)(1-r)(1-s)}{rs}\right) = \text{swap}\left( - \frac{1}{s} -  \frac{1}{r} \right) = -\frac{1}{uv} \notag
    \end{align}
    as we expect. Suppose that $\wjack(\mu) = -\wmnop(\mu)$ for all partitions of size $n-1$. By Theorem \ref{ratiotheorem}, given a partition $\lambda$ of size $n$, we know that
    \begin{align}
        \frac{\wjack(\lambda)}{\wjack(\mu)} = \frac{\wmnop(\lambda)}{\wmnop(\mu)} \notag
    \end{align}
    for any $\mu\subseteq\lambda$ of size $n-1$. Multiplying both sides by $\wjack(\mu) = -\wmnop(\mu)$ yields
    \begin{align}
        \wjack(\lambda) = -\wmnop(\lambda). \notag
    \end{align}
\end{proof}

\begin{acknowledgements}
    The authors would like to thank Golnaz Bahrami, Houcine Ben Dali, Cesar Cuenca, Maciej Do\l\k{e}ga, Martijn Kool, Dan Romik, Chris Sinclair, and Kayla Wright for helpful conversations. The first author was supported by NSF grant DMS-2039316.
\end{acknowledgements}

%%%%%%%%%%%%%%%%%%%%%%%%%%%%%%%%%%%%%%%%%%%%%%%%%%%%%%%%%%%%%%%

%\newpage

%\printbibliography
\bibliography{maincopy.bib}

\end{document}